\documentclass{siamltex}
\usepackage{amsmath, amsfonts, amscd, epsfig, amssymb}

\newcommand{\N}{\mathbb{N}}
\newcommand{\R}{\mathbb{R}}
\newcommand{\s}[1]{\mathcal{#1}}
\newcommand{\set}[1]{\left\{ #1 \right\}}
\newcommand{\setp}[2]{\left\{ #1 \: \middle| \: #2 \right\}}
\newcommand{\tr}{{\rm tr}}
\newcommand{\br}[1]{\left\langle #1 \right\rangle}
\newcommand{\paren}[1]{\left( #1 \right)}

\renewcommand{\vec}{\mbox{vec }}
\renewcommand{\epsilon}{\varepsilon}

\title{Computing expected transition events in reducible Markov
Chains\footnotemark[1]}

\author{Brian  D. Ewald\footnotemark[2] \and Jeffrey Humpherys\footnotemark[3]\ \footnotemark[4] \and Jeremy M. West\footnotemark[3]\ \footnotemark[5]}

\begin{document}

\maketitle

\renewcommand{\thefootnote}{\fnsymbol{footnote}}

\footnotetext[1]{Special thanks to C. Shane Reese for his useful conversations. We also thank our referees for their excellent comments and suggestions.} \footnotetext[2]{Department of Mathematics, Florida State University, Tallahassee, FL 32306.}
\footnotetext[3]{Department of Mathematics, Brigham Young University, Provo, UT 84602.} \footnotetext[4]{Work partially supported by NSF Grant No. DMS-6007721.} 
\footnotetext[5]{Work partially supported by BYU Graduate Research Fellowship Award.}

\renewcommand{\thefootnote}{\arabic{footnote}}

\begin{abstract}
We present a closed-form, computable expression for the expected number of times any transition event occurs during the transient phase of a reducible Markov chain. Examples of events include time to absorption, number of visits to a state, traversals of a particular transition, loops from a state to itself, and arrivals to a state from a particular subset of states. We give an analogous expression for time-average events, which describe the steady-state behavior of reducible chains as well as the long-term behavior of irreducible chains.
\end{abstract}

\begin{keywords}
Markov chains, reducible matrices, generalized inverses
\end{keywords}

\begin{AMS}
60J22, 15A09
\end{AMS}

\section{Introduction}\label{introduction-section}

In this paper we present a method for computing the expectation of transition events that occur during the transient phase of a reducible Markov chain using the Hadamard product and a generalized $(1,2)$-inverse; see \cite{campbell-meyer}. Some examples of transition events are the time to absorption, the number of visits to a state, the number of traversals of a transition, the number of loops from a state to itself, and the arrivals to a state from a particular subset of states.

Meyer \cite{meyer-siam-review} showed among other things that the group generalized inverse, a special case of the Drazin inverse, can be used to determine (i) the expected number of visits to any transient state, and (ii) the probability of absorption into a particular state. Whereas Meyer's method determines the expected number of occurrences of state events, our method computes the expected number of occurrences of transition events. Nonetheless, by summing over all transitions entering a given state, one can also compute the expected number of occurrences of state events, therefore, counting transitions is not an alternative to counting states, but a generalization. Indeed, we provide examples of quantities that are easily determined using transition information but are cumbersome, at best, to compute with state events.

We describe a transition event with a matrix, which we call a {\em mask}, where the entries of the mask are the weights assigned to each transition of the Markov chain. For example, the mask for the expected time to absorption assigns a unit weight to each transition that leaves a transient state and zero to all other transitions. The main result of this paper is a closed-form, computable expression for the expected value of a transition event; see Theorem \ref{reducible-chains-theorem}.

In \S \ref{main-results-section}, we define the random variables associated with a mask and give an expression for the expectation of these random variables on reducible Markov chains. Next we examine the time-average of a mask, which yields both the long-term behavior of irreducible chains and the steady-state behavior of reducible chains. In \S \ref{examples-section}, we give examples of masks. In \S \ref{computation-section}, we address the numerical issues of conditioning, stability, and complexity. In \S \ref{simulations-section}, we compare expectations computed using our results to a Monte Carlo simulation as a verification of our methods.

\section{Main Results}\label{main-results-section}

In this section we develop a closed-form expression for the expectation of a transition event. After dispensing with the preliminaries, we treat the absorbing chain case, that is, where the ergodic classes are single states. Next, we generalize to reducible chains. Finally, we show how masks can be used to determine the steady-state behavior of reducible chains and the time-average of irreducible chains.

\subsection{Preliminaries}\label{preliminaries-subsection}

To avoid confusion with the transition matrix $T$ we denote the transpose of a matrix by $A^*$. Let $A \odot B$ denote the Hadamard product, that is $(A \odot B)_{i,j} = A_{i,j}B_{i,j}$.

\begin{theorem}[see \cite{horn-johnson}]\label{hadamard-conversion-theorem}
Let $x \in \R^n$ and $A, B \in \R^{m \times n}$ be given and let $D = \diag(x)$. Then $(ADB^*)_{i,i} = \left[(A \odot B)x\right]_i$. 
\end{theorem}

In this paper we consider finite, stationary (temporally homogeneous) Markov chains, denoted $X_k$; see for example \cite{durrett}. Here, $\s{S} = \set{s_1, \ldots, s_n}$ is the state space. If $\mu \in \R^n$ is stochastic, that is $\mu_i \geq 0$ and $\|\mu\|_1 = 1$, then $P_\mu$ is the unique probability measure on $\Omega = \s{S} \times \s{S} \times \cdots$ satisfying $P_\mu(X_0 = s_i) = \mu_i$ and having transition probabilities associated with the Markov chain $X_k$. Furthermore, $E_\mu$ is expectation with respect to $P_\mu$. The column-stochastic matrix $T \in \R^{n \times n}$ with entries
\begin{equation}
T_{i,j} = P(X_{k+1} = s_i \mid X_k = s_j)
\end{equation}
is the transition matrix. The $k$-step transition probabilities are found in $T^k$. To summarize,
\begin{equation}
P_{\mu}(X_k = s_i) = \left[T^k\mu\right]_i.
\end{equation}

A mask is a matrix $M \in \R^{n \times n}$ that describes the weights assigned to the transitions of a Markov chain. Here $M_{i,j}$ is the weight assigned to the transition from $s_j$ to $s_i$. The transition event for $M$ is the random variable whose value on any realization is the sum of the mask entries,
\begin{equation}\label{random-variable}
Y_M = \sum_{k=0}^\infty M_{X_{k+1}, X_k}.
\end{equation}

\begin{lemma}\label{expectation-of-M-lemma}
Given $M \in \R^{n \times n}$,
\begin{equation} 
	E_\mu M_{X_{k+1}, X_k} = \sum_{i=1}^n \left[(M \odot T)T^k\mu\right]_i. 
\end{equation}
\end{lemma}

{\em Proof.} By total probability,
\begin{eqnarray*}
E_\mu M_{X_{k+1}, X_k} 
	&=& \sum_{i,j=1}^n M_{i,j}P_\mu(X_{k+1}=s_i, X_k = s_j) \\
    &=& \sum_{i,j=1}^n M_{i,j}P_\mu(X_{k+1}=s_i | X_k = s_j)P_\mu(X_k=s_j)\\
    &=& \sum_{i,j=1}^n M_{i,j}T_{i,j}\left[T^k\mu\right]_j = \sum_{i=1}^n \left[(M \odot T)T^k \mu\right]_i. \qquad \endproof
\end{eqnarray*}

\subsection{Cumulative Events on Absorbing Chains}
\label{absorbing-chains-subsection}

Let $\s{A} \subset \s{S}$ denote the absorbing states. That is, $s_j \in \s{A}$ if $P(X_{k+1} = s_j \mid X_{k} = s_j) = 1$, or equivalently, $T_{j,j} = 1$. A Markov chain $X_k$ is absorbing if $\s{A} \neq \emptyset$ and there exists $k \in \N$ such that
\begin{equation}
P(X_k \in \s{A} \mid X_0 = s_j) > 0,\qquad j = 1, \ldots, n.
\end{equation}
In other words, an absorbing chain is a reducible chain in which the ergodic classes are single states; see for example \cite{campbell-meyer, kemeny-snell, langville-meyer, meyer}. Without loss of generality, the transition matrix of an absorbing chain assumes the form
\begin{equation}\label{block-form-equation}
T = \begin{bmatrix} A_T & 0 \\ B_T & I\end{bmatrix},
\end{equation}
where $A_T \in \R^{t \times t}$ and $t = n - |\s{A}|$ is the number of transient states. Thus, $A_T$ and $B_T$ are the transitions leaving the $t$ transient states. In particular, the diagonal entries of $A_T$ are strictly less than 1. Furthermore,
\begin{equation}\label{block-form-powers-equation}
T^k = \begin{bmatrix}
    A_T^k & 0 \\
    B_T\sum_{m=0}^{k-1} A_T^m & I
    \end{bmatrix}.
\end{equation}

\begin{lemma}[see \cite{meyer}]\label{bounded-spectral-radius-lemma}
If $T$ is the transition matrix of an absorbing Markov chain then the spectral radius of $A_T$ satisfies $\rho(A_T) < 1$. Moreover, $(I - A_T)^{-1}$ exists and
\begin{equation}
(I - A_T)^{-1} = \sum_{k = 0}^\infty A_T^k.
\end{equation}
\end{lemma}

\begin{lemma} \label{convergence-of-M-dot-T-times-Tk-lemma}
Let $M, T \in \R^{n \times n}$ be given, where $T$ is the transition matrix of an absorbing Markov chain. If $M_{j,j} = 0$ whenever $s_j \in \s{A}$ then
\begin{equation}
\sum_{k=0}^\infty (M \odot T)T^k = (M \odot T)\begin{bmatrix}(I - A_T)^{-1} & 0 \\ 0 & 0\end{bmatrix}.
\end{equation}
\end{lemma}

{\em Proof.} If $s_j \in \s{A}$ then $(M \odot T)_{i,j} = 0$ for $i=1,\ldots,n$. Using the block form (\ref{block-form-equation}) for $M$,
\[ M \odot T = \begin{bmatrix} A_M \odot A_T & 0 \\ B_M \odot B_T & 0 \end{bmatrix}. \]
Combining this with \eqref{block-form-powers-equation},
\[ 
(M \odot T)T^k = \begin{bmatrix}(A_M \odot A_T)A_T^k & 0 \\ (B_M \odot
        B_T)A_T^k & 0\end{bmatrix} = (M \odot T)\begin{bmatrix}A_T^k & 0 \\ 0 & 0 \end{bmatrix}.
\]
Hence,
\begin{align*}
\sum_{k=0}^\infty (M \odot T)T^k
    &= (M\odot T)\sum_{k=0}^\infty\begin{bmatrix}A_T^k & 0
        \\ 0 & 0\end{bmatrix}
    = (M\odot T)\begin{bmatrix}(I-A_T)^{-1}&0\\0&0\end{bmatrix}.
    \qquad \endproof
\end{align*}

Throughout the paper, let
\begin{equation}\label{Q-equation}
Q = \begin{bmatrix}I - A_T & 0 \\ 0 & 0 \end{bmatrix},\quad \mbox{ and } \quad Q^- = \begin{bmatrix}(I - A_T)^{-1} & 0 \\ 0 & 0 \end{bmatrix}.
\end{equation}
Note that $Q^-$ satisfies $(I - T)Q^-(I - T) = (I - T)$ and $Q^-(I - T)Q^- = Q^-$ so that $Q^-$ is a (1,2)-inverse of $I-T$; see for example \cite{campbell-meyer}. However, it is not always the case that $((I-T)Q^-)^* = (I-T)Q^-$ or that $(Q^-(I-T))^* = Q^-(I-T)$. Hence, $Q^-$ is neither the Moore-Penrose inverse, nor is it the Drazin inverse of $I-T$ since $I-T$ and $Q^-$ do not necessarily commute. However, it is straightforward to show that $Q^-$ is both the Moore-Penrose inverse and the Drazin inverse of $Q$.

\begin{theorem}\label{absorbing-chains-theorem}
Let $M,T \in \R^{n \times n}$ and $\mu \in \R^n$ be given, where $T$ is the transition matrix of an absorbing Markov chain and $\mu$ is stochastic. Set $D = \diag(Q^-\mu)$. If $M_{j,j} = 0$ for all $s_j \in \s{A}$ then the random variable \eqref{random-variable} has expectation
\begin{equation}\label{absorbing-expectation}
E_{\mu}Y_M = \tr(MDT^*).
\end{equation}
\end{theorem}

\begin{proof}
Suppose that $M_{i,j} \geq 0$ for all $i,j$ so that $Y_M$ is an increasing series. Then by the Monotone Convergence Theorem, see \cite{durrett}, we may exchange the order of summation and expectation,
\[
E_{\mu}Y_M = \sum_{k=0}^\infty E_{\mu}M_{X_{k+1}, X_k}.
\]
Applying Lemma \ref{expectation-of-M-lemma}, 
\[ 
E_{\mu}Y_M = \sum_{k=0}^\infty \sum_{i=1}^n \left[(M \odot T)T^k\mu\right]_i 
	= \sum_{i=1}^n \left[\sum_{k=0}^\infty (M \odot T)T^k\mu\right]_i.
\]
By Lemma \ref{convergence-of-M-dot-T-times-Tk-lemma} and Theorem \ref{hadamard-conversion-theorem},
\[
E_{\mu}Y_M = \sum_{i=1}^n \left[(M\odot T)Q^-\mu\right]_i = \tr(MDT^*).
\]

For general $M$, let $Z$ be the random variable $Z = \sum_{k=0}^\infty |M_{X_{k+1}, X_k}|.$ For all $m \in \N$,
\[
\left|\sum_{k=0}^m M_{X_{k+1}, X_k} \right| \leq \sum_{k=0}^\infty |M_{X_{k+1}, X_k}| = Z.
\]
The nonnegative case indicates that $E_\mu |Z| = E_\mu Z < \infty$ so that the Dominated Convergence Theorem allows us to exchange the order of summation with expectation. The remainder of the argument is identical to the nonnegative case.\qquad
\end{proof}

{\em Remark.} Theorem \ref{absorbing-chains-theorem} indicates that the condition $M_{j,j} = 0$ for $s_j \in \s{A}$ is sufficient to guarantee that $E_\mu |Y_M| < \infty$. This condition is practically necessary; if $s_j \in \s{A}$ satisfies $P_\mu(X_k = s_j) > 0$ for some $k \in \N$ then $M_{j,j} \neq 0$ implies that $E_\mu |Y_M| = \infty$. Thus, $M_{j,j} = 0$ is required of all absorbing states that are ``reachable.''

\subsection{Cumulative Events on Reducible Markov Chains}
\label{cumulative-events-on-reducible-chains-subsection}

We now generalize to a reducible Markov chain; see for example \cite{langville-meyer, meyer, meyer-siam-review}. We assume that the transition matrix $T$ is in canonical form
\begin{equation}\label{reducible-canonical-form}
T = \left[\begin{array}{cccc|cccc}
        T_{11} & 0 & \dots & 0 & 0 & 0 & \dots & 0 \\
        T_{21} & T_{22} & \dots & 0 & 0 & 0 & \dots & 0 \\
        \vdots & \vdots & \ddots & \vdots & \vdots & \vdots & \ddots & \vdots \\
        T_{r1} & T_{r2} & \dots & T_{rr} & 0 & 0 & \dots & 0 \\
        \hline
        T_{r+1,1} & T_{r+1,2} & \dots & T_{r+1,r} & T_{r+1,r+1}&0&\dots&0\\
        T_{r+2,1} & T_{r+2,2}& \dots & T_{r+2,r} & 0 & T_{r+2,r+2} &&0\\
        \vdots & \vdots & \ddots & \vdots & \vdots & \vdots & \ddots & \vdots  \\
        T_{m1} & T_{m2} & \dots & T_{mr} & 0 & 0 & \dots & T_{mm}
    \end{array}\right].
\end{equation}
The blocks $T_{11}$ through $T_{rr}$ are the transient classes and the blocks $T_{r+1,r+1}$ through $T_{mm}$ are the ergodic classes. Here, $\rho(T_{ii}) < 1$ for $i \leq r$. The ergodic classes of a reducible chain generalize the notion of an absorbing state to a collection of states. We generalize the block form (\ref{block-form-equation}) for $T$ to
\begin{equation}
T = \begin{bmatrix} A_T & 0 \\ B_T & E_T\end{bmatrix},
\end{equation}
where $A_T$ and $B_T$ correspond to the transient states and $E_T$ is block diagonal containing the ergodic classes. Let $\s{E}$ denote the set of ergodic states.

\begin{theorem}\label{reducible-chains-theorem}
Set $D = \diag(Q^-\mu)$. For a reducible Markov chain $T$, if $M_{i,j} = 0$ whenever $s_i, s_j \in \s{E}$, then the random variable \eqref{random-variable} has expectation
\begin{equation}\label{reducible-expectation}
E_\mu Y_M =  \tr(MDT^*).
\end{equation}
\end{theorem}

\begin{proof}
Since $\rho(T_{ii}) < 1$ for all the transient classes it follows that $\rho(A_T) < 1$ as in Lemma \ref{bounded-spectral-radius-lemma}. The condition $M_{i,j} = 0$ for $s_j \in \s{E}$ guarantees the result of Lemma \ref{convergence-of-M-dot-T-times-Tk-lemma}. With these results, the remainder of the proof is identical to the proof of Theorem \ref{absorbing-chains-theorem}.\qquad
\end{proof}

\subsection{Time-Average Events}
\label{time-average-events-subsection}

Masks may be used for a general Markov chain, although the random variable in \eqref{random-variable} may not converge. However, the limit 
\begin{equation}\label{cesaro-sums-equation}
G = \lim_{N \to \infty} \frac{1}{N}\sum_{k=0}^N T^k
\end{equation}
does exist. Set $S = I - T$ and let $G = I - SS^\#$, where $S^\#$ is the group generalized inverse, or Drazin inverse, of $S$; see for example \cite{campbell-meyer, langville-meyer, meyer-siam-review}. Then $G$ is the projector onto the null space of $S$ along the range of $S$.

\begin{theorem}\label{time-average-theorem}
Set $D = \diag(G\mu)$. Then for any stochastic $T$, the random variable
\begin{equation}\label{time-average-random-variable}
Y_M = \lim_{N \to \infty} \frac{1}{N} \sum_{k=0}^N M_{X_{k+1}, X_k}
\end{equation}
has expectation
\begin{equation}\label{time-average-expectation}
E_{\mu}Y_M = \tr(MDT^*).
\end{equation}
\end{theorem}

{\em Proof.} Let $\gamma = \max\setp{|M_{i,j}|}{1 \leq i,j \leq n}$.
Then for all $N \in \N$,
\[
\frac{1}{N} \sum_{k=0}^N M_{X_{k+1}, X_k} \leq 2\gamma
\]
so that we may apply the Dominated Convergence Theorem. This and the linearity of expectation give
\begin{align*}
E_\mu Y_m
    &= \lim_{N \to \infty} \frac{1}{N} \sum_{k=0}^N E_\mu
        M_{X_{k+1}, X_k} \\
    &= \sum_{i=1}^n \left[(M \odot T)\left(\lim_{N \to \infty} \frac{1}{N} \sum_{k=0}^N
        T^k\right)\mu\right]_i \\
    &= \sum_{i=1}^n \left[(M \odot T)G\mu\right]_i = \tr(MDT^*).\qquad \endproof
\end{align*}

{\em Remark.}
If $T$ is reducible, the value of $M$ on the transitions leaving transient states is irrelevant; the value of $Y_M$ on any realization depends only on the ergodic class that is entered. Thus, $Y_M$ represents the steady-state behavior of $T$ in this case. For example, in the case of an absorbing chain
\begin{equation}
E_\mu Y_M = \sum_{s_j \in \s{A}} P_\mu(X_k \to s_j)M_{j,j}.
\end{equation}
If we fix $s_j \in \s{A}$ and set $M_{j,j} = 1$ with all other entries zero, then $E_\mu Y_M$ is the probability of absorption into $s_j$ given the initial distribution $\mu$.

\section{Examples}
\label{examples-section} In this section we present examples of masks for determining some of the canonical quantities for reducible chains given by Meyer \cite{meyer-siam-review}. Then we give some novel examples that show the flexibility of transition events.

\subsection{Canonical Examples}
Meyer showed that $S^\#$ and $I - SS^\#$ contain the following values for absorbing chains:
\begin{romannum}

\item For $s_i \in \s{A}, \: (I - SS^\#)_{i,j}$ is the probability of being absorbed into state $s_i$ when initially in state $s_j$.

\item If $s_i, s_j \notin \s{A}$ then $S^\#_{i,j}$ is the expected number of times the chain will be in state $s_i$ when initially in state $s_j$.

\item The expected number of steps until absorption when initially in state $s_j \notin \s{A}$ is $\sum_{s_i \notin \s{A}} S^\#_{i,j}.$

\end{romannum}

For general reducible chains, Meyer suggests representing the ergodic class by a single state and using the above results to determine the same quantities. We can express these quantities in terms of transition events. Furthermore, we may do so on any reducible chain without having to convert to an absorbing representation. For any ergodic class $\s{E}_m$, let
\begin{equation}
M_{i,j} =   \begin{cases}
                1 & s_i \in \s{E}_m, \: s_j \notin \s{E}, \\
                0 & \mbox{otherwise}.
            \end{cases}
\end{equation}
Then $Y_M$ is 1 on any realization which enters $\s{E}_m$ and zero elsewhere. Thus, $E_\mu Y_M$ is the probability of absorption into $\s{E}_m$ which gives (i) for any reducible chain.

For (ii), given $s_h \notin \s{E}$, let
\begin{equation}
M_{i,j} =   \begin{cases}
                1 & i = h, \: s_j \notin \s{E}, \\
                0 & \mbox{otherwise}.
            \end{cases}
\end{equation}
Then $E_\mu Y_M$ is the expected number of arrivals at state $s_h$ given the initial distribution $\mu$. Setting $M_{i,j} = 1$ when $j = h$ instead of $i = h$ gives the expected number of departures from state $s_h$. These quantities may differ depending on the initial distribution.

To find (iii) let
\begin{equation}
M_{i,j} =   \begin{cases}
                1 & s_j \notin \s{E},   \\
                0 & \mbox{otherwise}.
            \end{cases}
\end{equation}
Then $E_\mu Y_M$ is the expected number of steps until absorption into some ergodic class.

In the next two examples, the desired quantities correspond directly with transitions, not states. Therefore, using transition events is more natural than attempting to reproduce the results using state information.

\subsection{Expected Path Lengths}
\label{expected-path-lengths-subsection}

Consider an object that moves between $n$ states with transition probabilities $T_{i,j}$ and suppose that $s_n$ is absorbing. Let $d(s_j, s_i)$ be the distance between $s_j$ and $s_i$ and set
\begin{equation}
M_{i,j} = \begin{cases} 0 & j = n \\ d(s_j,s_i) & \mbox{otherwise}.\end{cases}
\end{equation}
The random variable \eqref{random-variable} describes the distance traveled on any realization. If the initial position of the object has distribution $\mu$ then Theorem \ref{absorbing-chains-theorem} indicates that the expected distance traveled is given by \eqref{absorbing-expectation}.

\subsection{Application From Physics}
\label{application-from-physics-subsection}

Consider a hydrogen atom that is excited by an external energy source so that the atom's electron is perpetually changing energy states. Let $\{s_1, \ldots,  s_n\}$ be the various allowable energy levels and $T_{i,j}$ be the probability that the atom's electron moves from $s_j$ to $s_i$. Also, let $\mu$ be the distribution on the electron's initial position. To determine the portion of light emitted by the hydrogen atom that is in a particular range, say the visible light range, we set $M_{i,j} = 1$ for any transition that emits visible light and $M_{i,j} = 0$ otherwise. Then the portion of light that is visible in any realization is the time-average random variable given by \eqref{time-average-random-variable}. Applying Theorem \ref{time-average-theorem}, the expected portion of visible light is given by \eqref{time-average-expectation}.

\subsection{Composite Markov Chains}
\label{composite-markov-chains-subsection}

Suppose $T_1 \in \R^{n_1 \times n_1}$ and $T_2 \in \R^{n_2 \times n_2}$ are stochastic matrices. Let $T = T_1 \otimes T_2 \in \R^{n_1n_2 \times n_1n_2}$ be the Kronecker product of $T_1$ and $T_2$; see for example \cite{graham, horn-johnson, laub}. For simplicity, we label the entries of $T$ by $T_{(i_1, i_2),(j_1, j_2)}$ which represents the $i_2,j_2$ entry of the $i_1, j_1$ block of $T$ and is equal to $T_{(i_1, i_2), (j_1, j_2)} = [T_1]_{i_1, j_1}[T_2]_{i_2, j_2}$. It is straightforward to check that $T$ is also stochastic. Indeed, if $X_k$ is the Markov chain of $T_1$ with states $\{s_1, \ldots, s_{n_1}\}$ and $Y_k$ is the Markov chain of $T_2$ with states $\{t_1, \ldots, t_{n_2}\}$ then
\begin{equation}
T_{(i_1, i_2), (j_1, j_2)} = P(X_{k+1} = s_{i_1}, Y_{k+1} = t_{i_2} \mid X_k = s_{j_1}, Y_k = t_{j_2}).
\end{equation}
Similarly, given stochastic $\mu_1 \in \R^{n_1}$ and $\mu_2 \in \R^{n_2}$, the vector $\mu = \mu_1 \otimes \mu_2 \in \R^{n_1n_2}$ is stochastic and the same indexing scheme applies:
\begin{equation}
P_\mu(X_0 = s_{i_1}, Y_0 = t_{i_2}) = \mu_{(i_1, i_2)}.
\end{equation}
This generalizes in the obvious way for any finite number of transition matrices. 

Suppose $T_0$ is a competitive system and the states are ordered such that higher indices represent being closer to winning. Then the $p$-wise Kronecker product $T = T_0 \otimes \dots \otimes T_0$ represents competition between $p$ players taking turns. It is natural to ask what the expected number of lead changes is.

For clarity, let $p = 2$. We count a lead change if a player comes from behind and ends in the lead. If a tie is either created or broken on a turn, we count a half a lead change. The mask for two-player lead changes is given by
\begin{equation}
M_{(i_1,i_2),(j_1,j_2)} =
\begin{cases}
    0 & s_{j_1} \in \s{A} \mbox{ or } s_{j_2} \in \s{A} \\
    1 & j_2 < j_1 \mbox{ and } i_2 > i_1 \\
    1 & j_2 > j_1 \mbox{ and } i_2 < i_1 \\
    1/2 & j_2 = j_1 \mbox{ and } i_2 \neq i_1 \\
    1/2 & j_2 \neq j_1 \mbox{ and } i_2 = i_1 \\
    0 & \mbox{otherwise}.
\end{cases}
\end{equation}
When $s_{j_1} \in \s{A}$ the $(i_1, j_1)$ block is zero. For $s_{j_1} \notin \s{A}$ the $(i_1, j_1)$ block is
\begin{equation}
M_{(i_1, \cdot),(j_1, \cdot)} =
\begin{bmatrix}
0       & \ldots & 0    & 1/2   & 1     & \ldots & 1      & \\
\vdots  &        &\vdots&\vdots &\vdots &        & \vdots & \\
0       & \ldots & 0    & 1/2   & 1     & \ldots & 1      & \vdots \\
1/2     & \ldots & 1/2  & 0     & 1/2   & \ldots & 1/2    & 0 \\
1       & \ldots & 1    & 1/2   & 0     & \ldots & 0      & \vdots \\
\vdots  &        &\vdots&\vdots &\vdots &        & \vdots & \\
1       & \ldots & 1    & 1/2   & 0     & \ldots & 0      &
\end{bmatrix}.
\end{equation}

When $p > 2$, there are at least two natural ways to define a lead change. The first is to count a lead change whenever the player in the lead is passed by another. We count a half lead change for breaking or establishing a tie in the leading position. The second way to extend lead changes for $p > 2$ is to count the permutations in the players positions. For example, if $j_1 > j_2 > \dots > j_p$ and $i_1 < i_2 < \dots < i_p$, then this complete lead change gets a weight of $M_{(i_1, \ldots, i_p),(j_1, \ldots, j_p)} = 1 + \dots + p = p(p+1)/2$.

\section{Computation}
\label{computation-section}

In this section we discuss the conditioning of \eqref{reducible-expectation}. We then provide an algorithm and analyze its complexity and stability. 

\subsection{Conditioning}
\label{conditioning-subsection}

We take our definition of conditioning from Trefethen and Bau \cite[p. 90]{trefethen}. For a function $f:\R^n \to \R^m$, denote
\begin{equation}
	\delta f(x) = f(x + \delta x) - f(x),
\end{equation}
where $x, \delta x \in \R^n$. The relative condition number of $f$ at $x \in \R^n$ is
\begin{equation}
	\kappa(x) = \lim_{\delta \to 0}\sup_{\|\delta x\| \leq \delta} \left.\frac{\|\delta f(x)\|}{\|f(x)\|} \middle/ \frac{\|\delta x\|}{\|x\|} \right. = \lim_{\delta \to 0}\sup_{\|\delta x\| \leq \delta} \frac{\|\delta f(x)\|}{\|\delta x\|}\frac{\|x\|}{\|f(x)\|}.\label{eq:relative-condition-number}
\end{equation}
In this section, we give bounds on $\kappa = \sup_{x} \kappa(x)$ for the function $f(T, M, \mu) = \tr(MDT^*)$ defined by \eqref{reducible-expectation}. We treat this as three separate conditioning problems by analyzing the conditioning of $f$ with respect to each input $M, T,$ and $\mu$ individually; see, for example \cite[Lecture 18]{trefethen}. This affords an understanding of the sensitivity of \eqref{reducible-expectation} to perturbations in each input.

The factor $\|x\|/\|f(x)\|$ in \eqref{eq:relative-condition-number} is independent of the perturbation $\delta x$ and may be pulled outside the limit. Our general approach in bounding the condition number is to evaluate, for each input separately, the absolute condition number
\begin{equation}
	\lim_{\delta \to 0} \sup_{\|\delta x\| \leq \delta} \frac{\|\delta f(x)\|}{\|\delta x\|}
\end{equation}
and then multiply by $\|x\|/\|f(x)\|$ to obtain \eqref{eq:relative-condition-number}.

Although the one-norm is a natural choice for column-stochastic matrices, $M$ and $D$ are not stochastic and the trace in \eqref{reducible-expectation} corresponds more naturally to the Frobenius inner product on the space of matrices. Therefore, we give bounds on the condition number $\kappa$ in terms of the Frobenius norm $\|A\|_F = \sqrt{\tr(A^*A)}$. By Cauchy-Schwarz, $|\tr(A^*B)| \leq \|A\|_F\|B\|_F.$ Furthermore, the Frobenius norm satisfies the submultiplicative property, that is, $\|AB\|_F \leq \|A\|_F\|B\|_F$. Therefore,
\begin{equation}\label{eq:trace-bound}
	|\tr(MDT^*)| = |\tr(T^*MD)| \leq \|T\|_F\|M\|_F\|D\|_F.
\end{equation}

\begin{theorem}\label{Frobenius-conditioning-theorem}
	Set
	\begin{equation}
		\kappa = \frac{\|M\|_F\|T\|_F\|(I - A_T)^{-1}\|_2}{|\tr(MDT^*)|}.
	\end{equation}
	The relative condition numbers for the expectation of transition events have the following bounds:
\begin{subequations}
	\begin{eqnarray}
		\kappa_M &\leq& \kappa, \label{eq:kappa-M}  \\
		\kappa_T &\leq& \kappa(1 + \|T\|_F\|(I - A_T)^{-1}\|_2), \label{eq:kappa-T} \\
		\kappa_\mu &\leq& \kappa. \label{eq:kappa-mu}
	\end{eqnarray}
\end{subequations}
\end{theorem}

\begin{proof}
	Recall from Theorem~\ref{reducible-chains-theorem} and \eqref{Q-equation} that $D = \diag(\nu)$, where $\nu = Q^-\mu$. Since $\mu$ is stochastic and $\|\cdot \|_2 \leq \|\cdot \|_1$ we obtain the bound $\|\mu\|_2 \leq \|\mu\|_1 = 1$. Therefore,
\begin{equation}\label{eq:D-and-nu-bounds}
	\|D\|_F = \paren{\sum_{i=1}^n d_{ii}^2}^{1/2} = \|\nu\|_2 = \|Q^- \mu \|_2 \leq \|(I - A_T)^{-1}\|_2.
\end{equation}
	For $\kappa_M$, fix $T$ and $\mu$ and treat $f(M) = \tr(MDT^*)$ as a function of $M$ only. We remark that $D$ is independent of $M$. Therefore, for a perturbation $\delta M$ of $M$ we obtain
	\begin{eqnarray*}
		\|\delta f(M)\|_F &=& |\tr((M+\delta M)DT^*) - \tr(MDT^*)| \\
		 	&=& |\tr(\delta MDT^*)| \leq \|\delta M\|_F\|T\|_F\|(I - A_T)^{-1}\|_2
	\end{eqnarray*}
	by \eqref{eq:trace-bound} and \eqref{eq:D-and-nu-bounds}. Hence,
	\[
		\lim_{\delta \to 0}\sup_{\|\delta M\|_F \leq \delta} \frac{\|\delta f(M)\|_F}{\|\delta M\|_F} \leq \|T\|_F\|(I - A_T)^{-1}\|_2.
	\]
	Multiplying by $\|M\|_F/\|f(M)\|_F = \|M\|_F/|\tr(MDT^*)|$ we obtain \eqref{eq:kappa-M}.

	The matrix $D$ depends on both $T$ and $\mu$. We denote by $D_{T + \delta T}$ and $D_{\mu + \delta \mu}$ the diagonal matrix obtained from $T + \delta T$ and $\mu + \delta \mu$, respectively, and use a similar notation for $\nu$. A perturbation $\delta T$ of $T$ causes a perturbation in $Q^-$. If $\delta A_T$ is the submatrix of $\delta T$ corresponding to $A_T$, then
\begin{equation}
		\begin{bmatrix} (I - A_T - \delta A_T)^{-1} & 0 \\ 0 & 0 \end{bmatrix} = (Q - \delta Q)^-,
\end{equation}
where $\delta Q$ is $\delta A_T$ padded with zeros. In the limit as $\delta \to 0$, $\|\delta A_T\|_F \leq \|\delta T\|_F \leq \delta$ implies that the inverse $(I - A_T - \delta A_T)^{-1}$ exists. Note that
\begin{equation}
	\nu = Q^-\mu = \begin{bmatrix} (I - A_T)^{-1} & 0 \\ 0 & 0 \end{bmatrix}\mu = \begin{bmatrix} \tilde \nu \\ 0 \end{bmatrix},
\end{equation}
where $\tilde \nu \in \R^t$ is the transient portion of $\nu$. If $\tilde \mu$ and $\delta \tilde\nu_T$ also represent the transient portions of $\mu$ and $\delta \nu_T$, respectively, then
\begin{equation}
	(I - A_T - \delta A_T)(\tilde \nu + \delta \tilde \nu_T) = \tilde \mu.
\end{equation}
Since $(I - A_T)\tilde \nu = \tilde \mu$, it follows that
\begin{equation}\label{d-nu-d-T}
	\delta\tilde\nu_T = (I - A_T - \delta A_T)^{-1}\delta A_T\tilde\nu.
\end{equation}
Consider $f(T) = \tr(MDT^*)$ as a function of $T$ only, where $M$ and $\mu$ are fixed. Then
\begin{eqnarray}
	\|\delta f(T)\|_F &=& |\tr(MD_{T+\delta T}(T + \delta T)^*) - \tr(MDT^*)|\nonumber \\ 
		&\leq& |\tr(M(D_{T + \delta T}-D)T^*)| + |\tr(MD_{T + \delta T}\delta T^*)|\nonumber \\
		&\leq& \|M\|_F\|D_{T + \delta T}-D\|_F\|T\|_F + \|M\|_F\|D_{T + \delta T}\|_F\|\delta T\|_F, \label{eq:f(T)-bound}
\end{eqnarray}
by \eqref{eq:trace-bound}. We require bounds on $\|D_{T + \delta T}-D\|_F$ and $\|D_{T + \delta T}\|_F$. In terms of $\nu$ we have $\|D_{T + \delta T} - D\|_F = \|\nu + \delta \nu_T - \nu\|_2 = \|\delta \nu_T\|_2.$ Since $\|\nu\|_2 = \|\tilde \nu\|_2$ and $\|\delta \nu_T\|_2 = \|\delta \tilde \nu_T\|_2$, applying \eqref{d-nu-d-T} yields
	\begin{equation}
	 \|\delta \nu_T\|_2 \leq \|(I - A_T - \delta A_T)^{-1}\|_2 \|\delta A_T\|_2 \|\nu\|_2.
	\end{equation}
	Clearly, $\|\delta A_T\|_2 \leq \|\delta T\|_2 \leq \|\delta T\|_F$. Combining this fact with \eqref{eq:D-and-nu-bounds} we obtain,
	\begin{equation}\label{eq:first-term-kappa-T} 
		\|D_{T + \delta T} - D\|_F = \|\delta \nu_T\|_2 \leq \|(I - A_T - \delta A_T)^{-1}\|_2 \|(I - A_T)^{-1}\|_2\|\delta T\|_F.
	\end{equation}
	We now turn our attention to $\|D_{T + \delta T}\|_F = \|\nu + \delta \nu_T\|_2 \leq \|\nu\|_2 + \|\delta \nu_T\|_2.$ Using \eqref{eq:D-and-nu-bounds} and \eqref{eq:first-term-kappa-T},
\begin{equation}
	\|D_{T + \delta T}\|_F \leq \|(I - A_T)^{-1}\|_2 + \|(I - A_T - \delta A_T)^{-1}\|_2\|(I - A_T)^{-1}\|_2\|\delta T\|_F.\label{eq:second-term-kappa-T}
\end{equation}
Putting \eqref{eq:f(T)-bound}, \eqref{eq:first-term-kappa-T}, and \eqref{eq:second-term-kappa-T} together, we have
\begin{subequations}
	\begin{eqnarray}
		\frac{\|\delta f(T)\|_F}{\|\delta T\|_F} &\leq& \|M\|_F\|(I - A_T)^{-1}\|_2 \\
		 	&+& \|M\|_F\|T\|_F\|(I-A_T-\delta A_T)^{-1}\|_2\|(I-A_T)^{-1}\|_2\label{eq:second-term} \\
			&+& \|M\|_F\|(I-A_T-\delta A_T)^{-1}\|_2\|(I-A_T)^{-1}\|_2\|\delta T\|_F.\label{eq:last-term}
	\end{eqnarray}
\end{subequations}
In the limit as $\delta \to 0$, \eqref{eq:last-term} is zero and $(I - A_T - \delta A_T)^{-1} = (I - A_T)^{-1}$ in \eqref{eq:second-term}, hence
\[
	\lim_{\delta \to 0}\sup_{\|\delta T\|_F \leq \delta} \frac{\|\delta f(T)\|_F}{\|\delta T\|_F} \leq \|M\|_F\|(I - A_T)^{-1}\|_2\left(1 + \|T\|_F\|(I-A_T)^{-1}\|_2\right).
\]
Multiplying by $\|T\|_F/\|f(T)\|_F = \|T\|_F/|\tr(MDT^*)|$ we obtain \eqref{eq:kappa-T}.

	Denote by $\delta \nu_\mu$ the change in $\nu$ due to a perturbation $\delta \mu$ of $\mu$. This satisfies
\begin{equation} 
	\nu + \delta \nu_\mu = Q^-(\mu + \delta \mu).
\end{equation}
By multiplying and canceling equal terms, we obtain
\begin{equation}\label{d-nu-d-mu}
	\delta{\nu}_\mu = Q^-\delta\mu.	
\end{equation}

	We now consider $f(\mu) = \tr(MDT^*)$ as a function of $\mu$, where $M$ and $T$ are fixed. Applying \eqref{eq:trace-bound} we obtain
	\[ \|\delta f(\mu)\|_F = |\tr(MD_{\mu + \delta \mu}T^*) - \tr(MDT^*)| \leq \|M\|_F\|T\|_F\|\delta \nu_\mu\|_2. \]
	Using \eqref{d-nu-d-mu} we have $\|\delta \nu_\mu\|_2 \leq \|(I - A_T)^{-1}\|_2\|\delta \mu\|_2.$ Thus,
	\begin{equation} 
		\lim_{\delta \to 0}\sup_{\|\delta \mu\| \leq \delta} \frac{\|\delta f(\mu)\|_F}{\|\delta \mu\|_2} \leq \|M\|_F\|T\|_F\|(I - A_T)^{-1}\|_2.\label{eq:mu-absolute-condition}
	\end{equation}
	Since $\mu$ is stochastic, $\|\mu\|_F = \|\mu\|_2 \leq \|\mu\|_1 = 1$. Hence, multiplying \eqref{eq:mu-absolute-condition} by $\|\mu\|_F/\|f(\mu)\|_F \leq 1/|\tr(MDT^*)|$ we obtain \eqref{eq:kappa-mu}.\qquad
\end{proof}

Since $T$ is stochastic, $\|T\|_F \leq \sqrt{n}$. In all the examples in \S \ref{examples-section}, $\|M\|_F$ is no more than order $n^2$. Therefore, the magnitude of $\kappa$ depends primarily on two factors: $\|(I - A_T)^{-1}\|_2$ and $|\tr(MDT^*)|$. As $I - A_T$ becomes singular, $\|(I - A_T)^{-1}\|_2$ is unbounded. In this case, the conditioning may be poor, which is to be expected since the conditioning of the linear system $(I - A_T)\nu = \mu$ is also poor.

The conditioning may also be poor if $\tr(MDT^*)$ is close to zero, particularly when $\|M\|_F\|T\|_F\|(I - A_T)^{-1}\|_2$ is relatively large. As the trace is a summation, cancelation of large magnitude terms with opposite signs results in poor conditioning. However, in all the examples in \S \ref{examples-section}, $M$ is nonnegative. Since $D$ and $T$ are always nonnegative, cancellation is not a problem in this case, although the order of summation may affect roundoff errors; see \cite[p. 63]{higham}.

Even when $M$ is nonnegative, $\tr(MDT^*)$ may be small due to orthogonality. Recall that $\tr(A^*B)$ is the Frobenius inner product on $\R^{m \times n}$. Therefore, $\tr(MDT^*) = \tr(DT^*M) = \br{TD, M}_F = \|TD\|_F\|M\|_F\cos \theta$ where $\theta$ is the angle between $TD$ and $M$. If these matrices are nearly orthogonal, the condition number may be large. This orthogonality often results from measuring events that are very unlikely to occur.

The quadratic dependence on $\|T\|_F\|(I - A_T)^{-1}\|_2$ in the upper bound for $\kappa_T$ is to be expected since $E_\mu Y_M = \tr(MDT^*)$ depends on $T$ in two places: the product $DT^*$ and the computation of $\nu$.
	
		
\subsection{Implementation}
\label{implementation-subsection}

In this section we provide an algorithm for computing \eqref{reducible-expectation}. As before, $\tilde \mu$ and $\tilde \nu$ are the first $t$ entries of $\mu$ and $\nu$, respectively, where $t$ is the number of transient states. Let $M_j \cdot T_j$ denote the standard inner product of the $j^{th}$ columns of $M$ and $T$. Then \eqref{reducible-expectation} may be expressed as
\begin{equation}\label{eq:modified-expectation}
	\tr(MDT^*) = \sum_{i=1}^n [ (M \odot T) \nu]_i = \sum_{j=1}^t \nu_j M_j \cdot T_j.
\end{equation}

\addtocounter{theorem}{1}
{\sc Algorithm \thetheorem.} \label{algorithm} The following computes \eqref{eq:modified-expectation} for the inputs $T, M$, and $\mu$ where $T$ is in canonical form \eqref{reducible-canonical-form}.
\begin{enumerate}
\item Solve $(I - A_T)\tilde \nu = \tilde \mu$ by forming the $QR$ factorization of $(I - A_T)$ using Householder relfections; see, for example \cite{higham, trefethen}.
\item Compute the first $t$ columns of $R = TD$, where $D = \diag(\nu)$ by scaling the $j^{th}$ column of $T$ by $\nu_j$.
\item Compute $\psi = \sum_{j=1}^t M_j \cdot R_j.$
\end{enumerate}
\vspace{4pt}

We refer to Steps 1-2 as the {\em setup}. This portion of the algorithm depends only on $T$ and $\mu$. Furthermore, Step 3 depends only on $M$. If several transition events are to be determined for the same chain and initial distribution, the setup need only be computed once.

\emph{Remark.} The matrix $I - A_T$ is invertible and diagonally dominant by columns. Gaussian Elimination on such a system requires no pivots and is stable~\cite{higham}. However, the theoretical bounds for Gaussian Elimination are insufficient to provide satisfactory bounds for Algorithm~\ref{algorithm} beyond $n \approx 2300$. It is well-known that Gaussian Elimination generally performs much better in practice than numerical analysis suggests~\cite{higham}. This does not change the asymptotic complexity of Algorithm~\ref{algorithm} but does improve the constants.


\subsection{Complexity}
\label{complexity-subsection}

Recall that $I - A_T \in \R^{t \times t}$, where $t$ is the number of transient states. It is well known that the temporal complexity of Step 1 is $O(t^3)$ and the spatial complexity is $O(t^2)$; see, for example~\cite{demmel, higham, trefethen}. Steps 2 and 3 both have temporal and spatial complexity $O(nt)$. Therefore, the setup requires $O(t^3 + nt)$ time and $O(nt)$ space. Once the setup is completed, \eqref{eq:modified-expectation} may be computed in $O(nt)$ time and space for each mask representing a transition event.


\subsection{Stability}
\label{stability-subsection}
In this section we give bounds on the backward errors introduced in the computation of Algorithm~\ref{algorithm}. We rely on the notation of Higham~\cite{higham}. In particular, let $u$ denote the unit roundoff and let
\begin{equation}\label{eq:gamma-definition}
	\gamma_k = \frac{ku}{1 - ku}, \qquad\mbox{and}\qquad\tilde \gamma_k = \frac{cku}{1 - cku},
\end{equation}
where $c$ is a small integer constant independent of $k$. The following result is useful in manipulating bounds involving $\gamma_k$.
\begin{lemma}[see~{\cite[pp. 67]{higham}}]\label{lemma:gamma-rules}
If $|\delta| \leq \gamma_k$ and $|\epsilon| \leq \gamma_j$ then $(1 + \delta)(1 + \epsilon) = (1 + \xi)$ where $|\xi| \leq \gamma_{k + j}$.
\end{lemma}

\begin{theorem}
Given $T$, $M$ and $\mu$ the value $\hat \psi$ computed by Algorithm~\ref{algorithm} is the exact solution for the inputs $T + \Delta T, M + \Delta M$, and $\mu$, where $\Delta T$ and $\Delta M$ satisfy the following column-wise bounds
\begin{equation}\label{eq:stability-theorem}
	\|\Delta T_j\|_2 \leq 2\sqrt{n}\tilde\gamma_{n^2}\|T_j\|_2, \qquad \mbox{and} \qquad \|\Delta M_j\|_2 \leq \frac{(1 + 2\sqrt{n})\tilde\gamma_{n^2}}{\sqrt{1 - 4\sqrt{n}\tilde\gamma_{n^2}}} \|M_j\|_2,
\end{equation}
provided $1 - 4\sqrt{n}\tilde\gamma_{n^2} > 0$.
\end{theorem}

\begin{proof}
The computed solution obtained in Step 1 satisfies the following column-wise backward error bounds~\cite[p. 361]{higham}:
\[
	(I - A_T - \Delta A_T)\hat \nu = \tilde \mu + \Delta \tilde \mu, \quad \mbox{where } \|\Delta {A_T}_j\|_2 \leq \tilde\gamma_{n^2} \|(I - A_T)_j\|_2, \quad 1 \leq j \leq t,
\]
Since $T_j$ is stochastic, $1 = \|T_j\|_1 \leq \sqrt{n}\|T_j\|_2$, hence 
\[
	\|(I - A_T)_j\|_2 \leq 1 + \|T_j\|_2 \leq (\sqrt{n} + 1)\|T_j\|_2 \leq 2\sqrt{n}\|T_j\|_2.
\]
Setting
\[ 
	\Delta T = \begin{bmatrix} \Delta A_T & 0 \\ 0 & 0 \end{bmatrix},
\]
we obtain the bound
\begin{equation}\label{eq:delta-T-bound}
	\|\Delta T_j\|_2 = \|\Delta {A_T}_j\|_2 \leq 2\sqrt{n}\tilde\gamma_{n^2}\|T_j\|_2.
\end{equation}

Since $D$ is diagonal, the computation in Step 2 to produce the matrix $R = TD$ involves only a single multiplication in each entry of $R$. Therefore, the computed result satisfies,
\[
	\hat R_{i,j} = (1 + \delta_{i,j})\hat \nu_jT_{i,j}, \qquad |\delta_{i,j}| \leq u, \quad 1 \leq i \leq n,\, 1 \leq j \leq t,
\]
where $\delta_{i,j}$ is the relative error caused by roundoff in the multiplication $\hat \nu_jT_{i,j}$. Step 3 is the inner product of two $nt \times 1$ vectors: $\vec M \cdot \vec \hat R$. The computed result satisfies the following bound on backward errors,
\[ 
	\hat \psi = \sum_{j=1}^t\sum_{i=1}^n (1+\epsilon_{i,j})M_{i,j}\hat R_{i,j} = \sum_{j=1}^t \hat \nu_j \sum_{i=1}^n (1+\epsilon_{i,j})(1 + \delta_{i,j}) M_{i,j}T_{i,j},
\]
where $\epsilon_{i,j}$ is the backward error of the $(i,j)$ entry that results from the computation of the inner product and satisfies $|\epsilon_{i,j}| \leq \gamma_{nt}$. This error bound is independent of the order of summation; the bounds may be improved by a careful ordering of the terms \cite[p. 63]{higham}. Since $|\delta_{i,j}| \leq u \leq \gamma_1$, Lemma~\ref{lemma:gamma-rules} guarantees that $(1 + \epsilon_{i,j})(1 + \delta_{i,j}) = (1 + \xi_{i,j})$ where $|\xi_{i,j}| \leq \gamma_{nt+1}$. To obtain \eqref{eq:stability-theorem}, we require a perturbation $\Delta M$ satisfying 
\[
	\sum_{i=1}^n (1+\xi_{i,j})M_{i,j}T_{i,j} = \sum_{i=1}^n (M + \Delta M)_{i,j}(T + \Delta T)_{i,j}, \qquad 1 \leq j \leq t.
\]
Recall that $\Delta T$ was fixed above when solving the system $(I - A_T)\tilde \nu = \tilde \mu$. Canceling the term $M_{i,j}T_{i,j}$ from the summation and regrouping,
\begin{equation}\label{eq:delta-M-system}
	\sum_{i=1}^n \paren{\xi_{i,j}M_{i,j}T_{i,j} - M_{i,j}\Delta T_{i,j}} = \sum_{i=1}^n \Delta M_{i,j}(T_{i,j} + \Delta T_{i,j}), \qquad 1 \leq j \leq t,
\end{equation}
Let $\xi_j$ be the $j^{th}$ column of the matrix $\xi = (\xi_{i,j})$. For each $j$, the left hand side of \eqref{eq:delta-M-system} is the scalar quantity 
\[
	b_j = (\xi_j \odot M_j) \cdot T_j - \Delta T_j \cdot M_j,
\]
where, $\xi_j \odot M_j$ is the Hadamard, or entry-wise product. The system \eqref{eq:delta-M-system} is equivalent to $(T_j + \Delta T_j) \cdot \Delta M_j = b_j$, which, for nonzero $T_j + \Delta T_j$, has as a solution
\begin{equation}\label{eq:delta-M-solution}
	\Delta M_j = \frac{b_j}{\|T_j + \Delta T_j\|_2^2}(T_j + \Delta T_j).
\end{equation}
Using our bound on $\Delta T$, Cauchy-Schwarz guarantees
\begin{eqnarray}
	\|T_j + \Delta T_j\|_2^2 &=& \|T_j\|_2^2 + 2T_j \cdot \Delta T_j + \|\Delta T_j\|_2^2 \geq \|T_j\|_2^2 - 2\|T_j\|_2\|\Delta T_j\|_2\nonumber \\
	 &\geq& \|T_j\|_2^2 - 4\sqrt{n}\tilde\gamma_{n^2}\|T_j\|_2^2 = (1 - 4\sqrt{n}\tilde\gamma_{n^2}) \|T_j\|_2^2 > 0,\label{eq:T-plus-delta-T}
\end{eqnarray}
under the assumption $1 - 4\sqrt{n}\tilde\gamma_{n^2} > 0$. Therefore, the computed $\hat \psi$ is the exact solution \eqref{eq:modified-expectation} for the inputs $M + \Delta M, T + \Delta T$, and $\mu$. In \eqref{eq:delta-T-bound} we gave bounds for $\Delta T$. By Cauchy-Schwarz,
\begin{eqnarray*}
	\|\Delta M_j\|_2 &=& \frac{|b_j|\|T_j + \Delta T_j\|_2}{\|T_j + \Delta T_j\|_2^2} \leq \frac{|(\xi_j \odot M_j) \cdot T_j| + |M_j \cdot \Delta T_j|}{\|T_j + \Delta T_j\|_2} \\
	&\leq& \frac{\gamma_{nt+1}\|M_j\|_2\|T_j\|_2 + \|M_j\|_2\|\Delta T_j\|_2}{\|T_j + \Delta T_j\|_2} \leq \frac{\gamma_{nt+1} + 2\sqrt{n}\tilde\gamma_{n^2}}{\sqrt{1 - 4\sqrt{n}\tilde\gamma_{n^2}}}\|M_j\|_2,
\end{eqnarray*}
by \eqref{eq:delta-T-bound} and \eqref{eq:T-plus-delta-T} and the observation $\|T_j\|_2 \leq \|T_j\|_1 = 1$, since $T_j$ is stochastic. Finally, the bounds $t \leq n-1$ and $n \geq 1$ imply that $nt+1 \leq n^2$, so $\gamma_{nt+1} \leq \gamma_{n^2} \leq \tilde\gamma_{n^2}$ and we obtain \eqref{eq:stability-theorem}.\qquad
\end{proof}

\emph{Remark.} For fixed $n$, \eqref{eq:stability-theorem} simplifies to
\[ 
	\frac{(1 + 2\sqrt{n})\tilde\gamma_{n^2}}{\sqrt{1 - 4\sqrt{n}\tilde\gamma_{n^2}}} \leq \frac{4\sqrt{n}\tilde\gamma_{n^2}}{1 - 4\sqrt{n}\tilde\gamma_{n^2}} = O(\sqrt{n}\tilde\gamma_{n^2}),
\]
as $u \to 0$. The quantity $\Delta M_j$ obtained in \eqref{eq:delta-M-solution} is the solution to the optimization problem
\[
\begin{array}{lll}
	\mbox{minimize} & \: & \|\Delta M_j\|_2 \\
	\mbox{subject to} & &  (T_j + \Delta T_j) \cdot \Delta M_j = b_j.
\end{array}
\]

\section{Simulations}\label{simulations-section}

We conducted a numerical study by computing expectations and comparing them to a Monte Carlo simulation. We used the game Chutes and Ladders (or Snakes and Ladders), which is characterized by a substantial number of states (82) and exhibits a gradual drift towards the absorbing state combined with occasional large jumps. Furthermore, this game is a good illustration of composite Markov chains as discussed in \S \ref{composite-markov-chains-subsection}. The MATLAB script used for computing expectations and the code for the simulations can be found in \cite{jeffs-website}. We simulated the following events in 100 million games and determined the sample mean for each. The results are summarized in Table \ref{simulation-results-table}.

\begin{itemize}
\item Second-To-Last Square: This is the number of times that a player gets ``stuck'' on the second-to-last square for spinning a number larger than 1.
\item Large Ladder Traversal: The number of times a player traverses the largest ladder from square 28 to square 84.
\item Game Length: The number of turns in the game.
\end{itemize}
In addition to the above events the following were simulated for a two-player game.
\begin{itemize}
\item Lead Changes: The number of lead changes in the game as discussed in \S \ref{composite-markov-chains-subsection}.
\item First-player Advantage: This is the indicator event for the first player winning when both players finish on the same turn. In expectation, it is the probability that the first player wins by virtue of being the first player.
\item First-player Win Frequency: This is the indicator event for the first player winning. In expectation, this is the probability that the first player wins.
\end{itemize}

\begin{table}[tbh]
\caption{Comparison of Monte Carlo simulations with computed expectations. }
\label{simulation-results-table}
\begin{center}
\begin{tabular}{|c|c|c|c|}
\hline  & {\bf Sample}  & {\bf Computed} & {\bf Computation}\\
 {\bf Event} & {\bf Mean} & {\bf Expectation} & {\bf Time (sec)} \\
\hline \hline
\multicolumn{4}{|c|}{{\bf Single-Player Events}} \\
\hline \hline
Setup & & & 1.8(-3) \\
Second-To-Last Square & 1.2954 & 1.2958 &  1.3(-4) \\
Large Ladder & 0.5895 & 0.5896 & 1.0(-4) \\
Game Length & 39.596 & 39.598 & 2.9(-4) \\
\hline \hline
\multicolumn{4}{|c|}{{\bf Two-Player Events}} \\
\hline \hline
Setup & & &  2.5 \\
Second-To-Last Square & 1.1159 & 1.1166 & 8.1(-3) \\
Large Ladder & 0.8181 & 0.8180 & 3.2(-2) \\
Game Length & 26.513 & 26.513 & 3.1 \\
Lead Changes & 3.9679 & 3.9679 & 3.4 \\
First-Player Advantage & 0.0156 & 0.0156 & 6.2(-3) \\
First-Player Wins & 0.5078 & 0.5078 & 1.4(-1)  \\
\hline
\end{tabular}
\end{center}
\end{table}

As can be seen in Table \ref{simulation-results-table}, the results agree up to at least three significant digits in every case. The execution time for computing expectations, shown in the last column of the table, indicates that even moderately large problems can feasibly be solved using this approach; the 2-player Chutes and Ladders matrix has over 6500 rows. Parallelization would permit much larger problems, however, we expect that for large $n$, simulation will be faster, just as Monte Carlo integration is more efficient than quadrature for high-dimensional problems.

\newpage

\bibliographystyle{siam}
\bibliography{markov}

\begin{thebibliography}{10}

\bibitem{campbell-meyer}
{\sc S.~L. Campbell and C.~D. Meyer, Jr.}, {\em Generalized inverses of linear
  transformations}, Dover Publications Inc., New York, 1991.
\newblock Corrected reprint of the 1979 original.

\bibitem{demmel}
{\sc James~W. Demmel}, {\em Applied numerical linear algebra}, Society for
  Industrial and Applied Mathematics (SIAM), Philadelphia, PA, 1997.

\bibitem{durrett}
{\sc Richard Durrett}, {\em Probability: theory and examples}, Duxbury Press,
  Belmont, CA, second~ed., 1996.

\bibitem{graham}
{\sc Alexander Graham}, {\em Kronecker products and matrix calculus: with
  applications}, Ellis Horwood Ltd., Chichester, 1981.
\newblock Ellis Horwood Series in Mathematics and its Applications.

\bibitem{higham}
{\sc Nicholas~J. Higham}, {\em Accuracy and stability of numerical algorithms},
  Society for Industrial and Applied Mathematics (SIAM), Philadelphia, PA,
  second~ed., 2002.

\bibitem{horn-johnson}
{\sc Roger~A. Horn and Charles~R. Johnson}, {\em Topics in matrix analysis},
  Cambridge University Press, Cambridge, 1994.
\newblock Corrected reprint of the 1991 original.

\bibitem{jeffs-website}
{\sc Jeffrey Humpherys}, {\em
  http://math.byu.edu/\symbol{126}jeffh/mathematics/games/chutes/index.htm}.
\newblock Markov simulation code, 2008.

\bibitem{kemeny-snell}
{\sc John~G. Kemeny and J.~Laurie Snell}, {\em Finite {M}arkov chains},
  Springer-Verlag, New York, 1976.
\newblock Reprinting of the 1960 original, Undergraduate Texts in Mathematics.

\bibitem{langville-meyer}
{\sc Amy~N. Langville and Carl~D. Meyer}, {\em Google's {P}age{R}ank and
  beyond: the science of search engine rankings}, Princeton University Press,
  Princeton, NJ, 2006.

\bibitem{laub}
{\sc Alan~J. Laub}, {\em Matrix analysis for scientists \& engineers}, Society
  for Industrial and Applied Mathematics (SIAM), Philadelphia, PA, 2005.

\bibitem{meyer}
{\sc Carl Meyer}, {\em Matrix analysis and applied linear algebra}, Society for
  Industrial and Applied Mathematics (SIAM), Philadelphia, PA, 2000.
\newblock With 1 CD-ROM (Windows, Macintosh and UNIX) and a solutions manual
  (iv+171 pp.).

\bibitem{meyer-siam-review}
{\sc Carl~D. Meyer, Jr.}, {\em The role of the group generalized inverse in the
  theory of finite {M}arkov chains}, SIAM Rev., 17 (1975), pp.~443--464.

\bibitem{trefethen}
{\sc Lloyd~N. Trefethen and David Bau, III}, {\em Numerical linear algebra},
  Society for Industrial and Applied Mathematics (SIAM), Philadelphia, PA,
  1997.

\end{thebibliography}

\end{document}